\documentclass{amsart}
\usepackage{amsmath,amsthm,fixltx2e}
\usepackage{amsxtra}
\usepackage{amscd}
\usepackage{amsfonts}
\usepackage{amssymb}
\usepackage{eucal}

\newtheorem{theorem}{Theorem}[section]

\newtheorem{lem}[theorem]{Lemma}
\newtheorem{prop}[theorem]{Proposition}
\newtheorem{thm}[theorem]{Theorem}
\newtheorem{rem}[theorem]{Remark}

\newcommand{\nc}{\newcommand}
\nc\ol{\overline} \nc\wt{\widetilde} \nc\ul{\underline}

\nc{\ZZ}{{\mathbb Z}} \nc{\NN}{{\mathbb N}} \nc{\CC}{{\mathbb C}}
\nc{\QQ}{{\mathbb Q}} \nc{\CP}{{\mathbb {CP}}} \nc{\A}{{\mathcal A}}
\nc\U{{\mathcal U}}  \nc{\N}{{\mathcal N}} \nc{\E}{{\mathcal E}}
\nc{\sS}{{\mathbb S}} \nc{\Y}{{\mathcal Y}} \nc{\SSS}{{\mathcal S}}
\nc\D{{\mathfrak D}} \nc\dd{{\mathfrak d}}\nc{\MM}{{\mathbb M}}

\newcommand\h{\mathfrak h}
\newcommand{\gl}{\mathfrak{gl}}
\newcommand{\ssl}{\mathfrak{sl}}
\newcommand{\Sym}{\mathrm{Sym}}

\newcommand{\spa}{\mathrm{span}}
\newcommand{\tr}{\mathrm{tr}}
\newcommand{\diag}{\mathrm{diag}}

\newcommand{\ad}{\mathop{\rm ad}\nolimits}
\nc{\iso}{{\stackrel{\sim}{\longrightarrow}}}

\begin{document}

\author[Alexander Tsymbaliuk]{Alexander Tsymbaliuk}
 \address{A.~Tsymbaliuk: Simons Center for Geometry and Physics, Stony Brook, NY 11794, USA}
 \email{otsymbaliuk@scgp.stonybrook.edu}

\title[Classical limits]
  {Classical limits of quantum toroidal and affine Yangian algebras}

\begin{abstract}
  In this short article, we compute the \emph{classical limits} of
 the quantum toroidal and affine Yangian algebras of $\ssl_n$ by
 generalizing our arguments for $\gl_1$ from~\cite{T1}
 (an alternative proof for $n>2$ is given in~\cite{VV}).
 We also discuss some consequences of these results.
\end{abstract}

\maketitle

\section*{Introduction}

  The primary purpose of this note is to provide proofs for the
 description of the \emph{classical limits} of the algebras
 $\U^{(n)}_{q,d}$ and $\Y^{(n)}_{h,\beta}$ from~\cite{FT,TB}.
  Here $\U^{(n)}_{q,d}$ and $\Y^{(n)}_{h,\beta}$ are the quantum toroidal
 and the affine Yangian algebras of $\ssl_n$ (if $n\geq 2$) or $\gl_1$ (if $n=1$),
 while \emph{classical limits} refer to the limits of these
 algebras as $q\to 1$ or $h\to 0$, respectively.
  We also discuss the \emph{classical limits} of certain constructions for $\U^{(n)}_{q,d}$.

  The case $n=1$ has been essentially worked out in~\cite{T1}.
 In this note, we follow the same approach to prove the $n>1$ generalizations.
 While writing down this note, we found that the $n\geq 3$ case
 has been considered in~\cite{VV} long time ago
 (to deduce our Theorems~\ref{main 1} and~\ref{main 2},
 one needs to combine~\cite{VV} with~\cite{BGK}).
 Hence, the only essentially new case is $n=2$.
 Meanwhile, we expect our direct arguments to be applicable
 in some other situations of interest.

 This paper is organized as follows:

 $\bullet$
  In Section 1, we recall explicit definitions of the Lie algebras
 $\ddot{u}^{(n)}_d$ and $\ddot{y}^{(n)}_\beta$, whose universal enveloping
 algebras coincide with the \emph{classical limits} of $\U^{(n)}_{q,d}$ and $\Y^{(n)}_{h,\beta}$.
 We also recall the notion of $n\times n$ matrix algebras over the algebras
 of difference/differential operators on $\CC^\times$ and their central extensions, denoted by
 $\bar{\dd}^{(n)}_t$ and $\bar{\D}^{(n)}_s$, respectively.

 $\bullet$
  In Section 2, we establish two key isomorphisms relating the
 \emph{classical limit} Lie algebras $\ddot{u}^{(n)}_d, \ddot{y}^{(n)}_\beta$
 to the aforementioned Lie algebras $\bar{\dd}^{(n)}_{d^n}, \bar{\D}^{(n)}_{n\beta}$.

 $\bullet$
  In Section 3, we discuss the \emph{classical limits} of
 the following constructions for $\U^{(n)}_{q,d}\ (n\geq 2)$:

\noindent
  -- the \emph{vertical} and \emph{horizontal} copies of a quantum
 affine algebra $U_q(\widehat{\gl}_n)$ inside $\U^{(n)}_{q,d}$ from~\cite{FJMM},

\noindent
  -- the Miki's automorphism $\varpi: \U^{(n)}_{q,d}\iso \U^{(n)}_{q,d}$ from~\cite{M},

\noindent
  -- the commutative subalgebras $\A(s_0,\ldots,s_{n-1})$ of $\U^{(n),+}_{q,d}$ from~\cite{FT}.


\subsection*{Acknowledgments}
$\ $

  I would like to thank Boris Feigin who emphasized
 an importance of the \emph{classical limit} considerations.
 Special thanks are due to Benjamin Enriquez who provided valuable
 comments on~\cite{FT} and asked about the \emph{classical limits}
 of the constructions from the \emph{loc.cit.}
 The author is deeply indebted to the anonymous referee for useful comments
 on the first version of the paper.

  The author gratefully acknowledges support from the
 Simons Center for Geometry and Physics, Stony Brook University,
 at which most of the research for this paper was performed.
 This work was partially supported by the NSF Grant DMS--1502497
 and the SUNY Individual Development Award.


\section{Basic constructions}


\subsection{The quantum toroidal algebra $\U^{(n)}_{q,d}$ and the affine Yangian $\Y^{(n)}_{h,\beta}$}
$\ $

  For $n\in \NN$, set $[n]:=\{0,1,\ldots,n-1\}$ viewed as a set of
 $\mathrm{mod}\ n$ residues and $[n]^\times:=[n]\backslash\{0\}$.
 For $n\geq 2$, we set $a_{i,j}:=2\delta_{i,j}-\delta_{i,j+1}-\delta_{i,j-1}$
 and $m_{i,j}:=\delta_{i,j+1}-\delta_{i,j-1}$ for all $i,j\in [n]$.

\medskip
\noindent
 $\circ$
  Given $h,\beta\in \CC$, let $\Y^{(n)}_{h,\beta}$ be the
 affine Yangian of $\ssl_n$ (if $n\geq 2$) or $\gl_1$ (if $n=1$) as considered in~\cite{TB},
 where it was denoted by $\Y^{(n)}_{\beta-h,2h,-\beta-h}$.
 These are unital associative $\CC$-algebras generated by
 $\{x^{\pm}_{i,r}, \xi_{i,r}\}_{i\in [n]}^{r\in \ZZ_+}$
 (here $\ZZ_+:=\NN\cup\{0\}$) and with the defining relations as in~\cite[Sect. 1.2]{TB}.
 We will list these relations only for $h=0$,
 which is of main interest in the current paper.

\medskip
\noindent
 $\circ$
  Given $q,d\in \CC^\times$, let $\U^{(n)}_{q,d}$ be the
 quantum toroidal algebra of $\ssl_n$ (if $n\geq 2$) or $\gl_1$ (if $n=1$)
 as considered in~\cite{FT}
 but without the generators $q^{\pm d_1}, q^{\pm d_2}$ and with $\gamma^{\pm 1/2}=q^{\pm c/2}$.
 These are unital associative $\CC$-algebras generated by
 $\{e_{i,k}, f_{i,k}, h_{i,k}, c\}_{i\in [n]}^{k\in \ZZ}$
 and with the defining relations specified in~\cite{FT}.
 We note that algebras $\U^{(n)}_{\frac{d}{q},q^2,\frac{1}{dq}}$
 from~\cite[Sect. 1.1]{TB} are their central quotients.


\subsection{The Lie algebra $\ddot{u}^{(n)}_d$}
$\ $

  In the $q\to 1$ limit, all the defining relations of $\U^{(n)}_{q,d}$
 become of \emph{Lie type}. Therefore, the $q\to 1$ limit of
 $\U^{(n)}_{q,d}$ is isomorphic to the universal enveloping algebra $U(\ddot{u}^{(n)}_d)$.
  The Lie algebra $\ddot{u}^{(n)}_d$ is generated by
 $\{\bar{e}_{i,k}, \bar{f}_{i,k}, \bar{h}_{i,k}, \bar{c}\}_{i\in [n]}^{k\in \ZZ}$
 with $\bar{c}$ being a central element and the rest of the defining relations (u1--u7.2)
 to be given below in each of the 3 cases of interest: $n>2,\ n=2,\ \mathrm{and}\ n=1$.

 $\bullet$ For $n>2$, the defining relations are
\begin{equation}\tag{u1}\label{u1.1}
 [\bar{h}_{i,k}, \bar{h}_{j,l}]=ka_{i,j}d^{-km_{i,j}}\delta_{k,-l}\bar{c},
\end{equation}
\begin{equation}\tag{u2}\label{u2.1}
 [\bar{e}_{i,k+1}, \bar{e}_{j,l}]=d^{-m_{i,j}}[\bar{e}_{i,k}, \bar{e}_{j,l+1}],
\end{equation}
\begin{equation}\tag{u3}\label{u3.1}
 [\bar{f}_{i,k+1}, \bar{f}_{j,l}]=d^{-m_{i,j}}[\bar{f}_{i,k}, \bar{f}_{j,l+1}],
\end{equation}
\begin{equation}\tag{u4}\label{u4.1}
 [\bar{e}_{i,k}, \bar{f}_{j,l}]=\delta_{i,j}\bar{h}_{i,k+l}+k\delta_{i,j}\delta_{k,-l}\bar{c},
\end{equation}
\begin{equation}\tag{u5}\label{u5.1}
 [\bar{h}_{i,k}, \bar{e}_{j,l}]=a_{i,j}d^{-km_{i,j}}\bar{e}_{j,l+k},
\end{equation}
\begin{equation}\tag{u6}\label{u6.1}
 [\bar{h}_{i,k}, \bar{f}_{j,l}]=-a_{i,j}d^{-km_{i,j}}\bar{f}_{j,l+k},
\end{equation}
\begin{equation}\tag{u7.1}\label{u7.1.1}
 \sum_{\pi\in \Sigma_2}\ [\bar{e}_{i,k_{\pi(1)}}, [\bar{e}_{i,k_{\pi(2)}}, \bar{e}_{i\pm 1,l}]]=0
 \ \mathrm{and}\
 [\bar{e}_{i,k}, \bar{e}_{j,l}]=0\ \mathrm{for}\ j\ne i, i\pm 1,
\end{equation}
\begin{equation}\tag{u7.2}\label{u7.2.1}
 \sum_{\pi\in \Sigma_2}\ [\bar{f}_{i,k_{\pi(1)}}, [\bar{f}_{i,k_{\pi(2)}}, \bar{f}_{i\pm 1,l}]]=0
 \ \mathrm{and}\
 [\bar{f}_{i,k}, \bar{f}_{j,l}]=0\ \mathrm{for}\ j\ne i, i\pm 1.
\end{equation}

 $\bullet$ For $n=2$, the defining relations are
\begin{equation}\tag{u1}\label{u1.2}
 [\bar{h}_{i,k}, \bar{h}_{i,l}]=2k\delta_{k,-l}\bar{c},\
 [\bar{h}_{i,k}, \bar{h}_{i+1,l}]=-k(d^k+d^{-k})\delta_{k,-l}\bar{c},
\end{equation}
\begin{equation}\tag{u2}\label{u2.2}
 [\bar{e}_{i,k+1}, \bar{e}_{i,l}]=[\bar{e}_{i,k}, \bar{e}_{i,l+1}],\
 [\bar{e}_{i,k+2}, \bar{e}_{i+1,l}]-(d+d^{-1})[\bar{e}_{i,k+1}, \bar{e}_{i+1,l+1}]+[\bar{e}_{i,k}, \bar{e}_{i+1,l+2}]=0,
\end{equation}
\begin{equation}\tag{u3}\label{u3.2}
 [\bar{f}_{i,k+1}, \bar{f}_{i,l}]=[\bar{f}_{i,k}, \bar{f}_{i,l+1}],\
 [\bar{f}_{i,k+2}, \bar{f}_{i+1,l}]-(d+d^{-1})[\bar{f}_{i,k+1}, \bar{f}_{i+1,l+1}]+[\bar{f}_{i,k}, \bar{f}_{i+1,l+2}]=0,
\end{equation}
\begin{equation}\tag{u4}\label{u4.2}
 [\bar{e}_{i,k}, \bar{f}_{j,l}]=\delta_{i,j}\bar{h}_{i,k+l}+k\delta_{i,j}\delta_{k,-l}\bar{c},
\end{equation}
\begin{equation}\tag{u5}\label{u5.2}
 [\bar{h}_{i,k}, \bar{e}_{i,l}]=2\bar{e}_{i,l+k},\
 [\bar{h}_{i,k}, \bar{e}_{i+1,l}]=-(d^k+d^{-k})\bar{e}_{i+1,l+k},
\end{equation}
\begin{equation}\tag{u6}\label{u6.2}
 [\bar{h}_{i,k}, \bar{f}_{i,l}]=-2\bar{f}_{i,l+k},\
 [\bar{h}_{i,k}, \bar{f}_{i+1,l}]=(d^k+d^{-k})\bar{f}_{i+1,l+k},
\end{equation}
\begin{equation}\tag{u7.1}\label{u7.1.2}
 \sum_{\pi\in \Sigma_3}\ [\bar{e}_{i,k_{\pi(1)}}, [\bar{e}_{i,k_{\pi(2)}}, [\bar{e}_{i,k_{\pi(3)}}, \bar{e}_{i+1,l}]]]=0,
\end{equation}
\begin{equation}\tag{u7.2}\label{u7.2.2}
 \sum_{\pi\in \Sigma_3}\ [\bar{f}_{i,k_{\pi(1)}}, [\bar{f}_{i,k_{\pi(2)}}, [\bar{f}_{i,k_{\pi(3)}}, \bar{f}_{i+1,l}]]]=0.
\end{equation}

 $\bullet$ For $n=1$, the defining relations are
\begin{equation}\tag{u1}\label{u1.3}
 [\bar{h}_{0,k}, \bar{h}_{0,l}]=k(2-d^k-d^{-k})\delta_{k,-l}\bar{c},
\end{equation}
\begin{equation}\tag{u2}\label{u2.3}
 [\bar{e}_{0,k+3}, \bar{e}_{0,l}]-(1+d+d^{-1})[\bar{e}_{0,k+2}, \bar{e}_{0,l+1}]+
 (1+d+d^{-1})[\bar{e}_{0,k+1}, \bar{e}_{0,l+2}]-[\bar{e}_{0,k}, \bar{e}_{0,l+3}]=0,
\end{equation}
\begin{equation}\tag{u3}\label{u3.3}
 [\bar{f}_{0,k+3}, \bar{f}_{0,l}]-(1+d+d^{-1})[\bar{f}_{0,k+2}, \bar{f}_{0,l+1}]+
 (1+d+d^{-1})[\bar{f}_{0,k+1}, \bar{f}_{0,l+2}]-[\bar{f}_{0,k}, \bar{f}_{0,l+3}]=0,
\end{equation}
\begin{equation}\tag{u4}\label{u4.3}
 [\bar{e}_{0,k}, \bar{f}_{0,l}]=\bar{h}_{0,k+l}+k\delta_{k,-l}\bar{c},
\end{equation}
\begin{equation}\tag{u5}\label{u5.3}
 [\bar{h}_{0,k}, \bar{e}_{0,l}]=(2-d^k-d^{-k})\bar{e}_{0,l+k},
\end{equation}
\begin{equation}\tag{u6}\label{u6.3}
 [\bar{h}_{0,k}, \bar{f}_{0,l}]=-(2-d^k-d^{-k})\bar{f}_{0,l+k},
\end{equation}
\begin{equation}\tag{u7.1}\label{u7.1.3}
 \sum_{\pi\in \Sigma_3}\ [\bar{e}_{0,k_{\pi(1)}}, [\bar{e}_{0,k_{\pi(2)}+1}, \bar{e}_{0,k_{\pi(3)}-1}]]=0,
\end{equation}
\begin{equation}\tag{u7.2}\label{u7.2.3}
 \sum_{\pi\in \Sigma_3}\ [\bar{f}_{0,k_{\pi(1)}}, [\bar{f}_{0,k_{\pi(2)}+1}, \bar{f}_{0,k_{\pi(3)}-1}]]=0.
\end{equation}

 In the above relations $l,k,k_1,k_2,k_3\in \ZZ$ and $\Sigma_s$ is
the symmetric group on $s$ letters.


\subsection{The Lie algebra $\ddot{y}^{(n)}_\beta$}
$\ $

  All the defining relations of $\Y^{(n)}_{h,\beta}$ are well-defined
 for $h=0$ and become of \emph{Lie type}.
 Therefore, $\Y^{(n)}_{0,\beta}\simeq U(\ddot{y}^{(n)}_\beta)$ where
 the Lie algebra $\ddot{y}^{(n)}_\beta$ is generated by
 $\{\bar{x}^\pm_{i,r}, \bar{\xi}_{i,r}\}_{i\in [n]}^{r\in \ZZ_+}$
 with the defining relations (y1--y6) to be given below.
 The first two of them are independent of $n\in \NN:$
\begin{equation}\tag{y1}\label{y1}
 [\bar{\xi}_{i,r}, \bar{\xi}_{j,s}]=0,
\end{equation}
\begin{equation}\tag{y2}\label{y2}
 [\bar{x}^+_{i,r}, \bar{x}^-_{j,s}]=\delta_{i,j}\bar{\xi}_{i,r+s}.
\end{equation}

\noindent
 Let us now specify (y3--y6) in each of the 3 cases of interest: $n>2,\ n=2,\ \mathrm{and}\ n=1$.

 $\bullet$ For $n>2$, the defining relations are
\begin{equation}\tag{y3}\label{y3.1}
 [\bar{x}^{\pm}_{i,r+1}, \bar{x}^{\pm}_{j,s}]-[\bar{x}^{\pm}_{i,r}, \bar{x}^{\pm}_{j,s+1}]=
 -m_{i,j}\beta [\bar{x}^{\pm}_{i,r}, \bar{x}^{\pm}_{j,s}],
\end{equation}
\begin{equation}\tag{y4}\label{y4.1}
 [\bar{\xi}_{i,r+1}, \bar{x}^{\pm}_{j,s}]-[\bar{\xi}_{i,r}, \bar{x}^{\pm}_{j,s+1}]=
 -m_{i,j}\beta[\bar{\xi}_{i,r}, \bar{x}^{\pm}_{j,s}],
\end{equation}
\begin{equation}\tag{y5}\label{y5.1}
 [\bar{\xi}_{i,0}, \bar{x}^{\pm}_{j,s}]=\pm a_{i,j}\bar{x}^{\pm}_{j,s},
\end{equation}
\begin{equation}\tag{y6}\label{y6.1}
 \sum_{\pi\in \Sigma_2}\ [\bar{x}^{\pm}_{i,r_{\pi(1)}}, [\bar{x}^{\pm}_{i,r_{\pi(2)}}, \bar{x}^{\pm}_{i\pm 1,s}]]=0
 \ \mathrm{and}\
 [\bar{x}^{\pm}_{i,r}, \bar{x}^{\pm}_{j,s}]=0\ \mathrm{for}\ j\ne i, i\pm 1.
\end{equation}

 $\bullet$ For $n=2$, the defining relations are
\begin{multline}\tag{y3}\label{y3.2}
 [\bar{x}^{\pm}_{i,r+1}, \bar{x}^{\pm}_{i,s}]=[\bar{x}^{\pm}_{i,r}, \bar{x}^{\pm}_{i,s+1}],\\
 [\bar{x}^{\pm}_{i,r+2}, \bar{x}^{\pm}_{i+1,s}]-2[\bar{x}^{\pm}_{i,r+1}, \bar{x}^{\pm}_{i+1,s+1}]+
 [\bar{x}^{\pm}_{i,r}, \bar{x}^{\pm}_{i+1,s+2}]=
 \beta^2 [\bar{x}^{\pm}_{i,r}, \bar{x}^{\pm}_{i+1,s}],
\end{multline}
\begin{multline}\tag{y4}\label{y4.2}
 [\bar{\xi}_{i,r+1}, \bar{x}^{\pm}_{i,s}]=[\bar{\xi}_{i,r}, \bar{x}^{\pm}_{i,s+1}],\\
 [\bar{\xi}_{i,r+2}, \bar{x}^{\pm}_{i+1,s}]-2[\bar{\xi}_{i,r+1}, \bar{x}^{\pm}_{i+1,s+1}]+
 [\bar{\xi}_{i,r}, \bar{x}^{\pm}_{i+1,s+2}]=
 \beta^2 [\bar{\xi}_{i,r}, \bar{x}^{\pm}_{i+1,s}],
\end{multline}
\begin{equation}\tag{y5}\label{y5.2}
 [\bar{\xi}_{i,0}, \bar{x}^{\pm}_{j,s}]=\pm a_{i,j}\bar{x}^{\pm}_{j,s},\
 [\bar{\xi}_{i,1},\bar{x}^{\pm}_{i+1,s}]=\mp 2 \bar{x}^{\pm}_{i+1,s+1},
\end{equation}
\begin{equation}\tag{y6}\label{y6.2}
 \sum_{\pi\in \Sigma_3}\ [\bar{x}^{\pm}_{i,r_{\pi(1)}}, [\bar{x}^{\pm}_{i,r_{\pi(2)}}, [\bar{x}^{\pm}_{i,r_{\pi(3)}}, \bar{x}^{\pm}_{i+1,s}]]]=0.
\end{equation}

 $\bullet$ For $n=1$, the defining relations are
\begin{multline}\tag{y3}\label{y3.3}
 [\bar{x}^{\pm}_{0,r+3}, \bar{x}^{\pm}_{0,s}]-3[\bar{x}^{\pm}_{0,r+2}, \bar{x}^{\pm}_{0,s+1}]+
 3[\bar{x}^{\pm}_{0,r+1}, \bar{x}^{\pm}_{0,s+2}]-[\bar{x}^{\pm}_{0,r}, \bar{x}^{\pm}_{0,s+3}]=\\
 \beta^2([\bar{x}^{\pm}_{0,r+1}, \bar{x}^{\pm}_{0,s}]-[\bar{x}^{\pm}_{0,r}, \bar{x}^{\pm}_{0,s+1}]),
\end{multline}
\begin{multline}\tag{y4}\label{y4.3}
 [\bar{\xi}_{0,r+3}, \bar{x}^{\pm}_{0,s}]-3[\bar{\xi}_{0,r+2}, \bar{x}^{\pm}_{0,s+1}]+
 3[\bar{\xi}_{0,r+1}, \bar{x}^{\pm}_{0,s+2}]-[\bar{\xi}_{0,r}, \bar{x}^{\pm}_{0,s+3}]=\\
 \beta^2([\bar{\xi}_{0,r+1}, \bar{x}^{\pm}_{0,s}]-[\bar{\xi}_{0,r}, \bar{x}^{\pm}_{0,s+1}]),
\end{multline}
\begin{equation}\tag{y5}\label{y5.3}
 [\bar{\xi}_{0,0}, \bar{x}^{\pm}_{0,s}]=0,\ [\bar{\xi}_{0,1}, \bar{x}^{\pm}_{0,s}]=0,\
 [\bar{\xi}_{0,2}, \bar{x}^{\pm}_{0,s}]=\mp 2\beta^2\bar{x}^{\pm}_{0,s},
\end{equation}
\begin{equation}\tag{y6}\label{y6.3}
 \sum_{\pi\in \Sigma_3}\ [\bar{x}^{\pm}_{0,r_{\pi(1)}}, [\bar{x}^{\pm}_{0,r_{\pi(2)}}, \bar{x}^{\pm}_{0,r_{\pi(3)}+1}]]=0.
\end{equation}

In the above relations $s,r,r_1,r_2,r_3\in \ZZ_+$ and $i,j\in[n]$.

\begin{rem}
 For $\beta\ne 0$, the assignment
   $\bar{x}^{\pm}_{i,r}\mapsto \beta^r \bar{x}^{\pm}_{i,r},\ \bar{\xi}_{i,r}\mapsto \beta^r \bar{\xi}_{i,r}$
 (with $i\in [n],r\in\ZZ_+$) provides an isomorphism
 of Lie algebras $\ddot{y}^{(n)}_\beta\iso \ddot{y}^{(n)}_1$.
\end{rem}


\subsection{Difference operators on $\CC^\times$}
$\ $

  For $t\in \CC^\times$, define the algebra of \emph{$t$-difference operators on $\CC^\times$},
 denoted by $\dd_t$, to be the unital associative $\CC$-algebra
 generated by $Z^{\pm 1}, D^{\pm 1}$ with the defining relations
   $$Z^{\pm 1}Z^{\mp 1}=1,\ D^{\pm 1}D^{\mp 1}=1,\ DZ=t\cdot ZD.$$
  Define the associative algebra $\dd^{(n)}_t:=\MM_n\otimes \dd_t$,
 where $\MM_n$ stands for the algebra of $n\times n$ matrices
 (so that $\dd^{(n)}_t$ is the algebra of $n\times n$ matrices with values in $\dd_t$).
 We will view $\dd^{(n)}_t$ as a Lie algebra with the natural commutator-Lie bracket $[\cdot,\cdot]$.
 It is easy to check that the following formulas
 define two 2-cocycles $\phi^{(1)}, \phi^{(2)}\in C^2(\dd^{(n)}_t,\CC)$:
  $$\phi^{(1)}(M_1\otimes D^{k_1}Z^{l_1},M_2\otimes D^{k_2}Z^{l_2})=
    l_1t^{k_1l_1}\delta_{k_1,-k_2}\delta_{l_1,-l_2}\tr(M_1M_2),$$
  $$\phi^{(2)}(M_1\otimes D^{k_1}Z^{l_1},M_2\otimes D^{k_2}Z^{l_2})=
    k_1t^{k_1l_1}\delta_{k_1,-k_2}\delta_{l_1,-l_2}\tr(M_1M_2)$$
 for any $M_1,M_2\in \MM_n$ and $k_1,k_2,l_1,l_2\in \ZZ$.

  This endows
   $\bar{\dd}^{(n)}_t:=\dd^{(n)}_t\oplus \CC\cdot c^{(1)}_\dd\oplus \CC\cdot c^{(2)}_{\dd}$
  with the Lie algebra structure via
    $$[X+\lambda_1c^{(1)}_\dd+\lambda_2c^{(2)}_\dd,  Y+\mu_1c^{(1)}_\dd+\mu_2c^{(2)}_\dd]=
      XY-YX+\phi^{(1)}(X,Y)c^{(1)}_\dd+\phi^{(2)}(X,Y)c^{(2)}_\dd$$
  for any $X,Y\in \dd^{(n)}_t$ and $\lambda_1,\lambda_2,\mu_1,\mu_2\in \CC$.
  We also define a Lie subalgebra $\bar{\dd}^{(n),0}_t\subset \bar{\dd}^{(n)}_t$ via
  $$\bar{\dd}^{(n),0}_t:=
    \left\{\sum A_{k,l}D^kZ^l+\lambda_1c^{(1)}_\dd+\lambda_2c^{(2)}_\dd\in \bar{\dd}^{(n)}_t|
            \lambda_1,\lambda_2\in \CC, A_{k,l}\in \MM_n, \tr(A_{0,0})=0\right\}.$$


\subsection{Differential operators on $\CC^\times$}
$\ $

   For $s\in \CC$, define the algebra of \emph{$s$-differential operators on $\CC^\times$},
 denoted by $\D_s$, to be the unital associative $\CC$-algebra
 generated by $\partial, x^{\pm 1}$ with the defining relations
   $$x^{\pm 1}x^{\mp 1}=1,\ \partial x=x (\partial+s).$$
 Define the associative algebra $\D^{(n)}_s:=\MM_n\otimes \D_s$
 (so that $\D^{(n)}_s$ is the algebra of $n\times n$ matrices with values in $\D_s$).
 We will view $\D^{(n)}_s$ as a Lie algebra with the natural commutator-Lie bracket.
 Following~\cite[Formula (2.3)]{BKLY}, consider a 2-cocycle $\phi\in C^2(\D^{(n)}_s,\CC)$ given by
  $$\phi (M_1\otimes f_1(\partial)x^{l_1}, M_2\otimes f_2(\partial)x^{l_2})=
    \left\{
     \begin{array}{llr}
        \tr(M_1M_2)\cdot \sum_{a=0}^{l_1-1}f_1(as)f_2((a-l_1)s) & \text{if}\ l_1=-l_2>0\\
        -\tr(M_1M_2)\cdot \sum_{a=0}^{-l_1-1}f_2(as)f_1((a+l_1)s) & \text{if}\ l_1=-l_2<0\\
        0 & \text{otherwise}
     \end{array}
    \right.$$
  for arbitrary polynomials $f_1,f_2$ and any $M_1,M_2\in \MM_n,\ l_1,l_2\in \ZZ$.

  This endows $\bar{\D}^{(n)}_s:=\D^{(n)}_s\oplus \CC\cdot c_\D$
  with the Lie algebra structure via
   $$[X+\lambda c_\D, Y+\mu c_\D]=XY-YX+\phi(X,Y) c_\D$$
  for any $X,Y\in \D^{(n)}_s$ and  $\lambda,\mu\in \CC.$


\medskip

\section{Key isomorphisms}


\subsection{Main results}
$\ $

  Our first main result establishes a relation between the Lie algebras
 $\ddot{u}^{(n)}_d$ and $\bar{\dd}^{(n)}_t$.

\begin{thm}\label{main 1}
 For $d\in \CC^\times$ not a root of unity (we will denote this by $d\ne \sqrt{1}$), the assignment
  $$\bar{e}_{0,k}\mapsto E_{n,1}\otimes D^kZ,\
    \bar{f}_{0,k}\mapsto E_{1,n}\otimes Z^{-1}D^k,\
    \bar{h}_{0,k}\mapsto E_{n,n}\otimes D^k-d^{nk}E_{1,1}\otimes D^k+\delta_{0,k}c^{(1)}_{\dd},\
    \bar{c}\mapsto c^{(2)}_{\dd},$$
  $$\bar{e}_{i,k}\mapsto d^{(n-i)k}E_{i,i+1}\otimes D^k,\
    \bar{f}_{i,k}\mapsto d^{(n-i)k}E_{i+1,i}\otimes D^k,\
    \bar{h}_{i,k}\mapsto d^{(n-i)k}(E_{i,i}-E_{i+1,i+1})\otimes D^k$$
 (with $i\in [n]^\times, k\in \ZZ$) provides an isomorphism of Lie algebras
 $\theta^{(n)}_d: \ddot{u}^{(n)}_d\iso \bar{\dd}^{(n),0}_{d^n}$.
\end{thm}

  Our second main result establishes a relation between the Lie algebras
 $\ddot{y}^{(n)}_\beta$ and $\bar{\D}^{(n)}_s$.

\begin{thm}\label{main 2}
 For $\beta\ne 0$, the assignment
  $$\bar{x}^+_{0,r}\mapsto E_{n,1}\otimes \partial^rx,\
    \bar{x}^-_{0,r}\mapsto E_{1,n}\otimes x^{-1}\partial^r,\
    \bar{\xi}_{0,r}\mapsto E_{n,n}\otimes \partial^r-E_{1,1}\otimes (\partial+n\beta)^r+\delta_{0,r}c_{\D},$$
  $$\bar{x}^+_{i,r}\mapsto E_{i,i+1}\otimes (\partial+(n-i)\beta)^r,
    \bar{x}^-_{i,r}\mapsto E_{i+1,i}\otimes (\partial+(n-i)\beta)^r,
    \bar{\xi}_{i,r}\mapsto (E_{i,i}-E_{i+1,i+1})\otimes (\partial+(n-i)\beta)^r$$
 (with $i\in [n]^\times, r\in \ZZ_+$) provides an isomorphism of Lie algebras
 $\vartheta^{(n)}_\beta: \ddot{y}^{(n)}_\beta\iso \bar{\D}^{(n)}_{n\beta}$.
\end{thm}

  For $n=1$, these isomorphisms have been essentially established in~\cite{T1}.
 In the rest of this section, we adapt arguments from~\cite{T1} to
 prove the above results for $n\geq 2$.

\begin{rem}
  These two theorems played a crucial role in~\cite{TB}, while their proofs were missing.
 In the \emph{loc. cit.}, we considered the quotients $\ddot{u}^{(n)}_d/(\bar{c})$ and
 $\bar{\dd}^{(n),0}_{d^n}/(c^{(2)}_\dd)$ and had a different 2-cocycle.
 Nevertheless, Theorem 2.8 from~\cite{TB} is equivalent to the above Theorem~\ref{main 1}.
\end{rem}


\subsection{Proof of Theorem~\ref{main 1}}
$\ $

  It is straightforward to see that the assignment from
 Theorem~\ref{main 1} preserves all the defining relations (u1--u7.2),
 hence, it provides a Lie algebra homomorphism
 $\theta^{(n)}_d: \ddot{u}^{(n)}_d\to \bar{\dd}^{(n),0}_{d^n}$.
 We also consider the induced homomorphism
 $\ul{\theta}^{(n)}_d: \ul{\ddot{u}}^{(n)}_d\to \dd^{(n),0}_{d^n}$, where
 $\ul{\ddot{u}}^{(n)}_d:=\ddot{u}^{(n)}_d/(\bar{c},\sum_i \bar{h}_{i,0})$
 is a central quotient of $\ddot{u}^{(n)}_d$.
  Clearly, it suffices to show that $\ul{\theta}^{(n)}_d$ is an isomorphism.

  Let $Q$ be the root lattice of $\widehat{\ssl}_n$.
 The Lie algebras $\ul{\ddot{u}}^{(n)}_d$ and $\dd^{(n)}_{d^n}$ are $Q\times \ZZ$--graded via
  $$\deg(\bar{e}_{i,k})=(\alpha_i;k),\
    \deg(\bar{f}_{i,k})=(-\alpha_i;k),\
    \deg(\bar{h}_{i,k})=(0;k),$$
  $$\deg(E_{i,j}\otimes D^kZ^l)=(l\delta+(\alpha_1+\ldots+\alpha_{j-1})-(\alpha_1+\ldots+\alpha_{i-1});k),$$
 where $\alpha_0,\alpha_1,\ldots,\alpha_{n-1}$ are the simple positive roots of $\widehat{\ssl}_n$,
 while $\delta=\alpha_0+\ldots+\alpha_{n-1}$ is the minimal positive imaginary root.
 Note that $\ul{\theta}^{(n)}_d$ is $Q\times \ZZ$--graded, and
 it is easy to see that $\ul{\theta}^{(n)}_d$ is surjective for $d\ne \sqrt{1}$.
  Therefore, it suffices to prove
\begin{equation}\tag{$\dag$}
 \dim(\ul{\ddot{u}}^{(n)}_d)_{(\alpha;k)}\leq \dim(\dd^{(n),0}_{d^n})_{(\alpha;k)}
\end{equation}
 for any $\alpha\in Q,\ k\in \ZZ$.
 Note that
  $$\dim(\dd^{(n),0}_{d^n})_{(\alpha;k)}=
   \begin{cases}
      0  & \text{if}\ \ \alpha\ \mathrm{is\ nonzero\ and\ is\ not\ a\ root\ of}\ \widehat{\ssl}_n\\
      1  & \text{if}\ \ \alpha\ \mathrm{is\ a\ real\ root\ of}\ \widehat{\ssl}_n\\
      n  & \text{if}\ \ \alpha\in \ZZ\delta\ \mathrm{and}\ (\alpha;k)\ne (0;0)\\
     n-1 & \text{if}\ \ (\alpha;k)=(0;0)
   \end{cases}.$$

  For $\alpha\notin \ZZ\delta$, the inequality ($\dag$) can be proved analogously
 to~\cite[Proposition 3.2]{MRY}\footnote{\ The argument in the \emph{loc. cit.} used
 the extra relation $[\bar{e}_{j,a},\bar{e}_{j,b}]=0$ for any $j\in [n], a,b\in \ZZ$ (with $n>1$). However,
 this relation is a simple consequence of (u2).} by viewing $\ddot{u}^{(n)}_d$ as a module
 over the \emph{horizontal} subalgebra generated by
 $\{\bar{e}_{i,0}, \bar{f}_{i,0}, \bar{h}_{i,0}\}_{i\in[n]}$,
 which is isomorphic to $\widehat{\ssl}_n$.
  Hence, it remains to handle the case $\alpha=l\delta$.
 The case $l=0$ is obvious since $(\ul{\ddot{u}}^{(n)}_d)_{(0;k)}$ is spanned
 by $\{\bar{h}_{i,k}\}_{i\in [n]}$. For the rest of the proof, we can assume $l\in \NN$.

\begin{rem}
  For $n>2$, this step is different from the argument in~\cite[Sect. 13]{VV}, where the authors
 prove that $\ddot{u}^{(n)}_d$ is the universal central extension of $\dd^{(n),0}_{d^n}$
 by showing that the former does not admit non-split central extensions.
\end{rem}

  Let $\ul{\ddot{u}}^{(n),\geq}_d$ be the subalgebra of $\ul{\ddot{u}}^{(n)}_d$ generated by
 $\{\bar{e}_{i,k},\bar{h}_{i,k}\}_{i\in [n]}^{k\in \ZZ}$.
 It is isomorphic to an abstract Lie algebra generated by
 $\{\bar{e}_{i,k},\bar{h}_{i,k}\}_{i\in [n]}^{k\in \ZZ}$
 subject to the defining relations (u1,u2,u5,u7.1) with $\bar{c}=0$ and $\sum_i \bar{h}_{i,0}=0$.
 It suffices to show that $\dim(\ul{\ddot{u}}^{(n),\geq}_d)_{(l\delta;k)}\leq n$ for any $l\in \NN, k\in \ZZ$.

  Introduce the \emph{length $N$ commutator}:
   $[a_1;a_2;\ldots;a_{N-1};a_N]_N:=[a_1,[a_2,[\ldots[a_{N-1},a_N]\ldots]]]$.
  We say that this commutator \emph{starts from} $a_1$.
 The degree $(l\delta;k)$ subspace of $\ul{\ddot{u}}^{(n),\geq}_d$
 is spanned by length $ln$ commutators
 $[\bar{e}_{i_1,k_1};\ldots;\bar{e}_{i_{ln},k_{ln}}]_{ln}$
 such that $k_1+\ldots+k_{ln}=k$ and $\alpha_{i_1}+\ldots+\alpha_{i_{ln}}=l\delta$.
 Define
   $v^{(i,l)}_{a,b}:=[\bar{e}_{i,a};\bar{e}_{i+1,0};\ldots;\bar{e}_{i-2,0};\bar{e}_{i-1,b}]_{ln}.$
 Note that $v^{(i,l)}_{a,b}\in (\ul{\ddot{u}}^{(n),\geq}_d)_{(l\delta;a+b)}$ and
   $v^{(i,l)}_{a,b}\ne 0$ since $\ul{\theta}^{(n)}_d(v^{(i,l)}_{a,b})\ne 0$.
 Combining this with $\dim(\ul{\ddot{u}}^{(n)}_d)_{(l\delta-\alpha_{i_1};k-k_1)}\leq 1$,
 we see that $(\ul{\ddot{u}}^{(n),\geq}_d)_{(l\delta;k)}$ is spanned by
   $\{v^{(i,l)}_{a,k-a}\}_{i\in [n]}^{a\in \ZZ}$.
 It remains to show that the rank of this system is at most $n$.

\medskip
\noindent
 $\bullet$ \emph{Case $k=0$.}

  Define $v_1:=v^{(1,l)}_{0,0},\ldots,v_{n-1}:=v^{(n-1,l)}_{0,0}, v_n:=v^{(0,l)}_{1,-1}$
 and set $V(l;0):=\spa_{\CC}\langle v_1,\ldots,v_n \rangle$.
 We prove $v^{(i,l)}_{a,-a}\in V(l;0)$ for all $i\in [n], a\in \ZZ$
 by induction on $|a|$. The case $a=0$ follows from
\begin{equation}\tag{$\diamondsuit$}\label{diamond}
  v^{(0,l)}_{0,0}+v^{(1,l)}_{0,0}+\ldots+v^{(n-1,l)}_{0,0}=0,
\end{equation}
 which is obvious once the \emph{horizontal} subalgebra of $\ul{\ddot{u}}^{(n)}_d$
 is identified with $\ssl_n[Z,Z^{-1}]$.

  To proceed further, we need the following technical result based on non-degeneracy of the
 matrices $(a_{i,j}d^{km_{i,j}})_{i\in [n]}^{j\in [n]}$ (for $n>2$) and
 $(2\delta_{i,j}-(d^k+d^{-k})\delta_{i,j+1})_{i\in [2]}^{j\in [2]}$
 for any $d\ne \sqrt{1}, k\ne 0$.

\begin{lem}\label{shift operators}
 For any fixed $i\in [n], k\ne 0$, there exists an element
   $\bar{h}'_{i,k}\in \spa_\CC\langle \bar{h}_{0,k},\ldots,\bar{h}_{n-1,k} \rangle$
 such that $[\bar{h}'_{i,k}, \bar{e}_{j,l}]=\delta_{i,j}\bar{e}_{j,l+k}$
 for all $j\in [n],l\in \ZZ$.
\end{lem}

  First, we prove $v^{(i,l)}_{-1,1}\in V(l;0)$.
 Applying $\ad(\bar{h}'_{i,-1})\ad(\bar{h}'_{0,1})$
 to the equality~(\ref{diamond}), we get a sum of
 $l^2n$ length $ln$ commutators being zero.
  Among those, $l^2n-l+\delta_{i,0}$ belong to $V(l;0)$
 as they start either from $\bar{e}_{i',0}\ (i'\in [n])$ or $\bar{e}_{0,1}$.
 The remaining $l-\delta_{i,0}$ commutators start from $\bar{e}_{i,-1}$
 and therefore are multiples of $v^{(i,l)}_{-1,1}$.
 It remains to show that the sum of these $l-\delta_{i,0}$ terms is nonzero.
 For the latter, it suffices to verify that the image of this sum
 under $\ul{\theta}^{(n)}_d$ is nonzero, which is a straightforward computation
 based on the assumption $d\ne \sqrt{1}$.
 To prove $v^{(i,l)}_{1,-1}\in V(l;0)$, we apply
 $\ad(\bar{h}'_{i,1})\ad(\bar{h}'_{i+1,-1})$ to~(\ref{diamond}) and follow the same arguments.

  To perform the inductive step, we assume that
 $v^{(i,l)}_{a,-a}\in V(l;0)$ for all $i\in [n], |a|\leq N$
 and we shall prove $v^{(i,l)}_{\pm (N+1),\mp (N+1)}\in V(l;0)$.
 Applying $\ad(\bar{h}'_{i,\pm(N+1)})\ad(\bar{h}'_{i+2,\mp 1})\ad(\bar{h}'_{i+1,\mp N})$ to~(\ref{diamond}),
 we get a sum of $l^3n$ length $ln$ commutators being zero.
 By the induction hypothesis, all of them,
 except for those starting from $\bar{e}_{i,\pm (N+1)}$, belong to $V(l;0)$.
  The remaining $l(l-\delta_{n,2})$ terms are multiples of $v^{(i,l)}_{\pm (N+1),\mp (N+1)}$.
 For $(n,l)\ne (2,1)$, it is easy to see that the sum of their images
 under $\ul{\theta}^{(n)}_d$ is nonzero,
 implying $v^{(i,l)}_{\pm (N+1),\mp (N+1)}\in V(l;0)$.
 In the remaining case $(n,l)=(2,1)$, the inclusion
 $v^{(i,l)}_{\pm (N+1),\mp (N+1)}\in V(l;0)$ follows from the relation (u2).

 This completes our induction step. Hence,
  $(\ul{\ddot{u}}^{(n),\geq}_d)_{(l\delta;0)}=V(l;0)\Rightarrow \dim(\ul{\ddot{u}}^{(n),\geq}_d)_{(l\delta;0)}\leq n$.

\medskip
\noindent
 $\bullet$ \emph{Case $0<k<l$.}

 Define $v_1:=v^{(1,l)}_{0,k},\ldots,v_{n-1}:=v^{(n-1,l)}_{0,k}, v_n:=v^{(0,l)}_{0,k}$
 and set $V(l;k):=\spa_{\CC}\langle v_1,\ldots,v_n \rangle$.
 We claim that $v^{(i,l)}_{a,k-a}\in V(l;k)$ for any $i\in [n], a\in \ZZ$.
 We will prove this in three steps.

\noindent
 \underline{Step 1}: Proof of $v^{(i,l)}_{k,0}\in V(l;k)$ for any $i\in [n]$.

  Applying $\ad(\bar{h}'_{i,k})$ to the equality~(\ref{diamond}),
 we immediately get $v^{(i,l)}_{k,0}\in V(l;k)$.

\noindent
 \underline{Step 2}: Proof of $v^{(i,l)}_{a,k-a}\in V(l;k)$ for any $i\in [n], 0<a<k$.

  It is known that any degree $k$ symmetric polynomial in $\{x_j\}_{j=1}^l$
 is a polynomial in $\{\sum_j x^r_j\}_{r=1}^k$.
 Choose $P_{k,l}$ such that
 $\Sym (x_1x_2\cdots x_k)=P_{k,l}(\sum_j x_j,\ldots, \sum_j x_j^k)$.
  Define $L_{i;k,l}\in \mathrm{End}(\ul{\ddot{u}}^{(n),\geq}_d)$ via
 $L_{i;k,l}=P_{k,l}(\ad(\bar{h}'_{i,1}),\ldots,\ad(\bar{h}'_{i,k}))$.
  Applying $L_{i;k,l}$ to the equality~(\ref{diamond}), we get a sum
 of ${{l}\choose{k}}\cdot n$ length $ln$ commutators being zero.
 Each of these terms starts either from $\bar{e}_{i',0}\ (i'\in [n])$ or
 $\bar{e}_{i,1}$. In the former case the commutator belongs to
 $V(l;k)$, while in the latter case the commutator is a multiple of
 $v^{(i,l)}_{1,k-1}$.
  There are ${l-1}\choose{k-1}$ terms starting
 from $\bar{e}_{i,1}$ and the sum of their images under $\ul{\theta}^{(n)}_d$ is nonzero.
 Therefore, $v^{(i,l)}_{1,k-1}\in V(l;k)$.

  Applying the same arguments to the symmetric function
 $\Sym (x_1^ax_2\cdots x_{k-a+1})$, we analogously get
 $v^{(i,l)}_{a,k-a}\in V(l;k)$ for any $i\in [n], 0<a<k$.

\noindent
 \underline{Step 3}: Proof of $v^{(i,l)}_{a,k-a}\in V(l;k)$ for any $i\in [n]$ and $a\notin\{0,1,\ldots,k\}$.

  We prove $v^{(i,l)}_{-N,k+N}, v^{(i,l)}_{k+N,-N}\in V(l;k)$ for all
 $i\in [n], N\in \ZZ_+$ by induction on $N$. The case $N=0$ is clear.
 Assume $v^{(i,l)}_{a,k-a}\in V(l;k)$ for any $i\in [n], -N\leq a\leq k+N$.
  Applying $\ad(\bar{h}'_{i,-N-1})\ad(\bar{h}'_{i+2,1})\ad(\bar{h}'_{i+1,k+N})$ to~(\ref{diamond}),
 we get a sum of $l^3n$ length $ln$ commutators being zero.
 Each of these terms either belongs to $V(l;k)$ by the induction
 hypothesis or is a multiple of $v^{(i,l)}_{-N-1,k+N+1}$.
 There are $l(l-\delta_{n,2})$ summands of the latter form and the sum of their
 images under $\ul{\theta}^{(n)}_d$ is nonzero.
 This implies $v^{(i,l)}_{-N-1,k+N+1}\in V(l;k)$.

 To prove $v^{(i,l)}_{k+N+1,-N-1}\in V(l;k)$, we apply
 $\ad(\bar{h}'_{i,k+N+1})\ad(\bar{h}'_{i+2,-1})\ad(\bar{h}'_{i+1,-N})$ to~(\ref{diamond})
 and follow the same arguments.

\medskip
\noindent
 $\bullet$ \emph{Case of an arbitrary $k$.}

 It is clear that $L_{i;l,l}$ induces an isomorphism
   $(\ul{\ddot{u}}^{(n),\geq}_d)_{(l\delta;k')}\iso (\ul{\ddot{u}}^{(n),\geq}_d)_{(l\delta;k'+l)}$
 for any $k'\in \ZZ$. In particular,
   $\dim (\ul{\ddot{u}}^{(n),\geq}_d)_{(l\delta;k)}=\dim (\ul{\ddot{u}}^{(n),\geq}_d)_{(l\delta;k\ \mathrm{mod}\ l)}\leq n$,
 due to the previous two cases.  $\blacksquare$


\subsection{Proof of Theorem~\ref{main 2}}
$\ $

  It is straightforward to see that the assignment from
 Theorem~\ref{main 2} preserves all the defining relations (y1--y6),
 hence, it provides a Lie algebra homomorphism
 $\vartheta^{(n)}_\beta: \ddot{y}^{(n)}_\beta\to \bar{\D}^{(n)}_{n\beta}$.
 We also consider the induced homomorphism
 $\ul{\vartheta}^{(n)}_\beta: \ul{\ddot{y}}^{(n)}_\beta\to \D^{(n)}_{n\beta}$,
 where $\ul{\ddot{y}}^{(n)}_\beta:=\ddot{y}^{(n)}_\beta/(\sum_i \bar{\xi}_{i,0})$
 is a central quotient of $\ddot{y}^{(n)}_\beta$.
 Clearly, it suffices to show that $\ul{\vartheta}^{(n)}_\beta$ is an isomorphism.

  The Lie algebra $\ul{\ddot{y}}^{(n)}_\beta$ is $Q$--graded via
 $\deg_1(\bar{x}^\pm_{i,r})=\pm \alpha_i,\ \deg_1(\bar{\xi}_{i,r})=0$
 and $\ZZ_+$--filtered as a quotient of the free Lie algebra
 on $\{\bar{x}^{\pm}_{i,r},\bar{\xi}_{i,r}\}_{i\in [n]}^{r\in \ZZ_+}$ graded via
 $\deg_2(\bar{x}^\pm_{i,r})=r,\ \deg_2(\bar{\xi}_{i,r})=r$.
  The Lie algebra $\D^{(n)}_{n\beta}$ is also $Q$--graded via
 $\deg_1(E_{i,j}\otimes \partial^rx^l)=l\delta+(\alpha_1+\ldots+\alpha_{j-1})-(\alpha_1+\ldots+\alpha_{i-1})$
 and $\ZZ_+$--filtered with the filtration $\leq k$ subspace
 consisting of the finite sums
 $\sum_{0\leq i\leq k}^{j\in \ZZ} A_{i,j}\partial^ix^j$,
 where $A_{i,j}\in \MM_n$ and $\tr(A_{k,j})=0$ for any $j\in \ZZ$.
 Let $(\ul{\ddot{y}}^{(n)}_\beta)_{(\alpha;\leq k)}$ and
 $(\D^{(n)}_{n\beta})_{(\alpha;\leq k)}$ denote the subspaces of
 $\ul{\ddot{y}}^{(n)}_\beta$ and $\D^{(n)}_{n\beta}$, respectively,
 consisting of the degree $\alpha$ and filtration $\leq k$ elements.


  Note that
   $\ul{\vartheta}^{(n)}_\beta((\ul{\ddot{y}}^{(n)}_\beta)_{(\alpha;\leq k)})\subset (\D^{(n)}_{n\beta})_{(\alpha;\leq k)}$
 for any  $\alpha\in Q, k\in \ZZ_+$. Hence, we get linear maps
  $\ul{\vartheta}^{(n)}_{\beta;\alpha,k}:
   (\ul{\ddot{y}}^{(n)}_\beta)_{(\alpha;\leq k)}/(\ul{\ddot{y}}^{(n)}_\beta)_{(\alpha;\leq k-1)}
   \to (\D^{(n)}_{n\beta})_{(\alpha;\leq k)}/(\D^{(n)}_{n\beta})_{(\alpha;\leq k-1)}$.
 We claim that all the maps $\ul{\vartheta}^{(n)}_{\beta;\alpha,k}$ are isomorphisms.
 To prove this, it suffices to show that $\ul{\vartheta}^{(n)}_{\beta;\alpha,k}$ is surjective and
\begin{equation}\tag{$\ddag$}
 \dim(\ul{\ddot{y}}^{(n)}_\beta)_{(\alpha;\leq k)}-\dim(\ul{\ddot{y}}^{(n)}_\beta)_{(\alpha;\leq k-1)}
 \leq \dim(\D^{(n)}_{n\beta})_{(\alpha;\leq k)}-\dim(\D^{(n)}_{n\beta})_{(\alpha;\leq k-1)}
\end{equation}
 for any $\alpha\in Q, k\in \ZZ_+$.
 The right-hand side of ($\ddag$) can be simplified as follows:
  $$\dim(\D^{(n)}_{n\beta})_{(\alpha;\leq k)}-\dim(\D^{(n)}_{n\beta})_{(\alpha;\leq k-1)}=
   \begin{cases}
      0  & \text{if}\ \ \alpha\ \mathrm{is\ nonzero\ and\ is\ not\ a\ root\ of}\ \widehat{\ssl}_n\\
      1  & \text{if}\ \ \alpha\ \mathrm{is\ a\ real\ root\ of}\ \widehat{\ssl}_n\\
      n-\delta_{k,0}  & \text{if}\ \ \alpha\ \mathrm{is\ an\ imaginary\ root\ or\ zero}
   \end{cases}.$$

  For $\alpha\notin \ZZ\delta$, the inequality~($\ddag$) and the surjectivity of
 $\ul{\vartheta}^{(n)}_{\beta;\alpha,k}$ can be deduced in the same way as~($\dag$).
 Hence, it remains to handle the case $\alpha=l\delta$.
 The $l=0$ case is obvious since the degree $0$ subspace of $\ul{\ddot{y}}^{(n)}_\beta$
 is spanned by $\bar{\xi}_{i,r}$.
 For the rest of the proof, we can assume $l\in \NN$.

  Let $\ul{\ddot{y}}^{(n),\geq}_\beta$ be the subalgebra of $\ul{\ddot{y}}^{(n)}_\beta$
 generated by $\{\bar{x}^+_{i,r}, \bar{\xi}_{i,r}\}_{i\in [n]}^{r\in \ZZ_+}$.
 It is isomorphic to an abstract Lie algebra generated by
 $\{\bar{x}^+_{i,r}, \bar{\xi}_{i,r}\}_{i\in [n]}^{r\in \ZZ_+}$
 subject to the defining relations (y1,y3,y4,y5,y6) and $\sum_i \bar{\xi}_{i,0}=0$.
 It suffices to show that
   $\dim(\ul{\ddot{y}}^{(n), \geq}_\beta)_{(l\delta;\leq k)}-\dim(\ul{\ddot{y}}^{(n), \geq}_\beta)_{(l\delta;\leq k-1)}\leq n-\delta_{k,0}$
 and $\ul{\vartheta}^{(n)}_{\beta;l\delta,k}$ is surjective for any $l\in \NN, k\in \ZZ_+$.

\medskip
\noindent
 \underline{Case $n=2$.}

  The degree $l\delta$ subspace of $\ul{\ddot{y}}^{(2),\geq}_\beta$ is spanned by
 all length $2l$ commutators $[\bar{x}^+_{i_1,k_1};\ldots;\bar{x}^+_{i_{2l},k_{2l}}]_{2l}$
 such that $\alpha_{i_1}+\ldots+\alpha_{i_{2l}}=l\delta$. Define
 $w^{(i,l)}_{a,b}:=[\bar{x}^+_{i,a}; \bar{x}^+_{i+1,0};\ldots; \bar{x}^+_{i,0};\bar{x}^+_{i+1,b}]_{2l}$
 for any $i\in [2]$ and $a,b\in \ZZ_+$.
  Due to our description of the degree $l\delta-\alpha_i$ subspace
 of $\ul{\ddot{y}}^{(2),\geq}_\beta$, we see that $(\ul{\ddot{y}}^{(2),\geq}_\beta)_{l\delta}$
 is spanned by $\{w^{(i,l)}_{a,b}\}_{i\in [2]}^{a,b\in \ZZ_+}$.
 Moreover, $(\ul{\ddot{y}}^{(2),\geq}_\beta)_{(l\delta,\leq k)}$ is
 spanned by $\{w^{(i,l)}_{a,b}\}_{i\in [2]}^{a,b\in \ZZ_+}$ with $a+b\leq k$.
 Therefore, the inequality
 $\dim(\ul{\ddot{y}}^{(2), \geq}_\beta)_{(l\delta;\leq k)}-\dim(\ul{\ddot{y}}^{(2), \geq}_\beta)_{(l\delta;\leq k-1)}\leq 2-\delta_{k,0}$
 and the surjectivity of $\ul{\vartheta}^{(2)}_{\beta;l\delta,k}$
 follow from our next result:

\begin{prop}\label{estimate}
 Define $W(l;N):=\spa_\CC\langle w^{(i,l)}_{0,M}\rangle_{i\in [2]}^{0\leq M\leq N}$
 for $l\in \NN,\ N\in \ZZ_+$.

\noindent
 (a) We have $w^{(i,l)}_{a,b}\in W(l;a+b)$ for any $i\in [2], l\in \NN, a,b\in \ZZ_+$.

\noindent
 (b) The images of $\{\ul{\vartheta}^{(2)}_\beta(w^{(i,l)}_{0,N})\}_{i\in [2]}$
 in the quotient space $(\D^{(2)}_{2\beta})_{(l\delta;\leq N)}/(\D^{(2)}_{2\beta})_{(l\delta;\leq N-1)}$
 are linearly independent for any $l,N\in \NN$.
\end{prop}

\begin{proof}[Proof of Proposition~\ref{estimate}]
$\ $

(a) Our proof is based on the following simple equalities:
\begin{equation}\label{auxiliary 1}
 \sum_{i\in [2]} w^{(i,l)}_{0,0}=0,
\end{equation}
\begin{equation}\label{auxiliary 2}
 [H_3,\bar{x}^+_{i,r}]=\bar{x}^+_{i,r+1},\ [H_4,\bar{x}^+_{i,r}]=\bar{x}^+_{i,r+2},
\end{equation}
 where
 $H_3:=\frac{-1}{6\beta^2}\sum_{i\in [2]} \bar{\xi}_{i,3},\
  H_4:=\frac{-1}{12\beta^2}\sum_{i\in [2]} \bar{\xi}_{i,4}+\frac{1}{12}\sum_{i\in [2]} \bar{\xi}_{i,2}$.

\noindent
 $\circ$ \emph{Proof of $w^{(i,l)}_{1,b}\in W(l;1+b)$.}

  We prove this by induction on $b$.
 Applying $\ad(H_3)$ or $\ad(\frac{1}{2}\bar{\xi}_{i,1})$ to~(\ref{auxiliary 1}), we get
 $w^{(0,l)}_{1,0}+w^{(1,l)}_{1,0}\in W(l;1)$ and
 $w^{(i,l)}_{1,0}-w^{(i+1,l)}_{1,0}\in W(l;1)$, respectively.
 Hence, $w^{(i,l)}_{1,0}\in W(l;1)$, which is the basis of induction.
  To perform the inductive step, we assume
 $w^{(i,l)}_{1,b}\in W(l;1+b)$ for any $0\leq b\leq M$.
 In particular, $w^{(i,l)}_{1,M}=\sum_{j\in [2]}^{N\leq M+1} c_{j,N} w^{(j,l)}_{0,N}$
 for some $c_{j,N}\in \CC$.
  Applying $\ul{\vartheta}^{(2)}_\beta$ to this equality, we find
 $c_{i,M+1}=\frac{M+1-l}{M+1}, c_{i+1,M+1}=\frac{-l}{M+1}$.
 Hence $w^{(i,l)}_{1,M}-\frac{M+1-l}{M+1}w^{(i,l)}_{0,M+1}+\frac{l}{M+1}w^{(i+1,l)}_{0,M+1}\in W(l;M)$.
 Applying $\ad(H_3+\frac{1}{2}\bar{\xi}_{i+1,1})$ to this inclusion, we get
 $l w^{(i,l)}_{1,M+1}+\frac{l}{M+1} w^{(i+1,l)}_{1,M+1}\in W(l;M+2)$.
  For $M>0$, this yields $w^{(i,l)}_{1,M+1}\in W(l;M+2)$ as
 $w^{(i,l)}_{1,M+1}\in \spa_\CC\langle l w^{(j,l)}_{1,M+1}+\frac{l}{M+1} w^{(j+1,l)}_{1,M+1} \rangle_{j\in [2]}$.

  It remains to treat separately the case $M=0$. We can assume $l>1$ as the case $l=1$ is simple.
 Applying $\ad(H_3+\frac{1}{2}\bar{\xi}_{i,1})$ to the inclusion
 $w^{(i,l)}_{1,0}+(l-1)w^{(i,l)}_{0,1}+lw^{(i+1,l)}_{0,1}\in W(l;0)$,
 we get $2(l-1)w^{(i,l)}_{1,1}+w^{(i,l)}_{2,0}\in W(l;2)$.
  On the other hand, applying $\ad(H_4+\frac{1}{2}\bar{\xi}_{i,2})$ to~(\ref{auxiliary 1}),
 we find $w^{(i,l)}_{2,0}\in W(l;2)$. This implies $w^{(i,l)}_{1,1}\in W(l;2)$.

\noindent
 $\circ$ \emph{Proof of $w^{(i,l)}_{a,b}\in W(l;a+b)$ for $a>1$.}

  We prove this by induction on $a$. The base cases of induction $a=0,1$ have been already treated.
 To perform the inductive step, we assume $w^{(i,l)}_{a,b}\in W(l;a+b)$ for all $0\leq a\leq M$ and $b\in \ZZ_+$.
 In particular, $w^{(i,l)}_{M,b}=\sum_{j\in [2]}^{N\leq M+b} d_{j,N} w^{(j,l)}_{0,N}$ for some $d_{j,N}\in \CC$.
 Applying $\ad(H_3)$ to this equality and using the induction hypothesis,
 we immediately get $w^{(i,l)}_{M+1,b}\in W(l;M+b+1)$.

 (b) Straightforward computations yield
 $\ul{\vartheta}^{(2)}_\beta(w^{(1,l)}_{0,N})=2^{N-1}(E_{1,1}-E_{2,2})\otimes \partial^Nx^l + \mathrm{l.o.t}$,
 $\ul{\vartheta}^{(2)}_\beta(w^{(0,l)}_{0,N}+w^{(1,l)}_{0,N})=-2^{l-1}N\beta\cdot (E_{1,1}+E_{2,2})\otimes \partial^{N-1}x^l + \mathrm{l.o.t.},$
 where $\mathrm{l.o.t.}$ denote summands with lower power of $\partial$.
 The result follows.
\end{proof}

  This completes our proof of Theorem~\ref{main 2} for $n=2$.

\medskip
\noindent
 \underline{Case $n>2$.}

  The proof for $n>2$ is completely analogous and crucially uses the same
 equalities~(\ref{auxiliary 1}) and (\ref{auxiliary 2}); we leave details to the interested reader. $\blacksquare$


\medskip

\section{Consequences}


\subsection{Classical limits of the vertical and horizontal quantum affine $\gl_n$}
$\ $

  For $n\geq 2$, the algebra $\U^{(n)}_{q,d}$ contains two subalgebras
 $\dot{U}^{v}_q,\ \dot{U}^{h}_q$ isomorphic to the quantum affine $U_q(\widehat{\ssl}_n)$.
 Here $\dot{U}^{v}_q$ is generated by
 $\{e_{i,k},f_{i,k},h_{i,k},c\}_{i\in [n]^\times}^{k\in \ZZ}$,
 while $\dot{U}^{h}_q$ is generated by
 $\{e_{i,0},f_{i,0},h_{i,0}\}_{i\in [n]}$.
 The following result is obvious.

\begin{lem}\label{consequence 1}
  For $d\ne \sqrt{1}$, the isomorphism $\theta^{(n)}_d$ identifies
 the $q\to 1$ limits of the subalgebras $\dot{U}^{v}_q$ and $\dot{U}^{h}_q$
 with the universal enveloping algebras of
 $\ssl_n[D,D^{-1}]\oplus \CC\cdot c^{(2)}_{\dd}$ and
 $\ssl_n[Z,Z^{-1}]\oplus \CC\cdot c^{(1)}_{\dd}$, respectively.
\end{lem}

  According to~\cite{FJMM}, the algebra $\U^{(n)}_{q,d}$ also contains two
 Heisenberg subalgebras $\h^{v}$ and $\h^{h}$, which commute with
 $\dot{U}^{v}_q$ and $\dot{U}^{h}_q$, respectively.
 This yields two copies of the quantum affine
 $U_q(\widehat{\gl}_n)$ inside $\U^{(n)}_{q,d}$,
 which will be denoted by $\dot{U}^{v,'}_q$ and $\dot{U}^{h,'}_q$, respectively.

\begin{lem}\label{consequence 2}
  For $d\ne \sqrt{1}$, the isomorphism $\theta^{(n)}_d$ identifies the $q\to 1$ limits
 of the subalgebras $\dot{U}^{v,'}_q$ and $\dot{U}^{h,'}_q$ with the universal enveloping
 algebras of $\gl_n[D,D^{-1}]^0\oplus \CC\cdot c^{(2)}_{\dd}$ and
 $\gl_n[Z,Z^{-1}]^0\oplus \CC\cdot c^{(1)}_{\dd}$, where
 $\gl[D,D^{-1}]^0:=\ssl_n\otimes 1\oplus \underset{k\ne 0}\bigoplus \gl_n\otimes D^k$ and
 $\gl[Z,Z^{-1}]^0:=\ssl_n\otimes 1\oplus \underset{k\ne 0}\bigoplus \gl_n\otimes Z^k$.
\end{lem}

\begin{proof}[Proof of Lemma~\ref{consequence 2}]
$\ $

  (i) First, we recall the construction of $\h^{v}$ from~\cite[Sect. 2.2]{FJMM}.
  For any $k\ne 0$ and $i,j\in [n]$, define the constants
   $b_n(i,j;k):=
    \begin{cases}
       d^{-km_{i,j}} \frac{q^{ka_{i,j}}-q^{-ka_{i,j}}}{k(q-q^{-1})} & \text{if}\ n>2 \\
       \delta_{i,j}\cdot \frac{q^{2k}-q^{-2k}}{k(q-q^{-1})}-\delta_{i,j+1}\cdot(d^k+d^{-k})\frac{q^k-q^{-k}}{k(q-q^{-1})} & \text{if}\ n=2
    \end{cases}$,
  so that their $q\to 1$ limits are equal to
   $\bar{b}_n(i,j;k)=
    \begin{cases}
       a_{i,j}d^{-km_{i,j}} & \text{if}\ n>2 \\
       2\delta_{i,j}-(d^k+d^{-k})\delta_{i,j+1} & \text{if}\ n=2
    \end{cases}$.
  For any fixed $k\ne 0$, let $\{c_{i,k}\}_{i\in [n]}$ be a unique solution of the system
  $\sum_{i\in [n]} b_n(i,j;k)c_{i,k}=0$ for all $j\in [n]^\times$ with $c_{0,k}=1$.
  By construction, the subalgebra $\h^{v}$ is generated by $q^{c/2}$
  and the elements $\{h^{v}_k:=\sum_{i\in [n]} c_{i,k}h_{i,k}\}_{k\ne 0}$.
  The image of the $q\to 1$ limit of $h^{v}_k$ under $\theta^{(n)}_d$ equals
    $$H^v_k=\left(\bar{c}_{0,k}(E_{n,n}-d^{nk}E_{1,1})+\sum_{i=1}^{n-1}\bar{c}_{i,k}d^{(n-i)k}(E_{i,i}-E_{i+1,i+1})\right)\otimes D^k,$$
  where the constants $\{\bar{c}_{i,k}\}$ satisfy
  $\sum_{i\in [n]} \bar{b}_n(i,j;k)\bar{c}_{i,k}=0\ \mathrm{for\ all}\ j\in [n]^\times$ and $\bar{c}_{0,k}=1$.
   Hence, $H^v_k=\frac{1-d^{nk}}{n}\cdot I_n\otimes D^k$ with $I_n=\sum_{j=1}^n E_{j,j}$.
  It remains to notice that the Lie subalgebra of $\bar{\dd}^{(n),0}_{d^n}$
  generated by $\ssl_n[D,D^{-1}]\oplus \CC\cdot c^{(2)}_\dd$ and $\{I_n\otimes D^k\}_{k\ne 0}$
  is exactly $\gl_n[D,D^{-1}]^0\oplus \CC\cdot c^{(2)}_{\dd}$.

   (ii) According to~\cite{FJMM}, $\dot{U}^{h,'}_q$ is a preimage of $\dot{U}^{v,'}_q$
  under the Miki's automorphism $\varpi$. Combining (i) with Lemma~\ref{consequence 3} below,
  we get the description of $\theta^{(n)}_d(q\to 1\ \mathrm{limit\ of}\ \dot{U}^{h,'}_q)$.
\end{proof}


\subsection{Classical limit of the Miki's automorphism}
$\ $

  The natural `90 degree rotation' automorphism of $\U^{(1)}_{q,d}$ (due to Burban--Schiffmann)
 admits a generalization to the case of $\U^{(n)}_{q,d}$ with $n\geq 2$ (due to Miki).
\begin{thm}\cite{M}
 For $n\geq 2$, there exists an automorphism $\varpi$ of $\U^{(n)}_{q,d}$ such that
  $$\varpi(\dot{U}^{v}_q)=\dot{U}^{h}_q,\ \varpi(\dot{U}^{h}_q)=\dot{U}^{v}_q,\
    \varpi(c)=-\sum_{i\in [n]} h_{i,0},\ \varpi(\sum_{i\in [n]} h_{i,0})=c.$$
\end{thm}

  Our next result provides a description of the $q\to 1$ limit of $\varpi$,
 denoted by $\bar{\varpi}$, viewed as an automorphism of the
 universal enveloping algebra $U(\bar{\dd}^{(n),0}_{d^n})$.

\begin{lem}\label{consequence 3}
  $\bar{\varpi}$ is induced by an automorphism of the Lie algebra $\bar{\dd}^{(n)}_{d^n}$ defined via
\begin{equation}\tag{$\star$}\label{star}
 c^{(1)}_\dd \mapsto c^{(2)}_\dd,\ c^{(2)}_\dd \mapsto -c^{(1)}_\dd,\
 A\otimes D^kZ^l\mapsto d^{-nk}(-d)^{nl} A\otimes Z^{-k}D^l\ \ \forall\ A\in \MM_n,\ k,l\in \ZZ.
\end{equation}
\end{lem}

\begin{proof}[Proof of Lemma~\ref{consequence 3}]
$\ $

   It is easy to see that the formulas ($\star$) define a Lie algebra automorphism;
 we denote its restriction to $\bar{\dd}^{(n),0}_{d^n}$  by $\wt{\varpi}$.
  On the other hand, the action of $\varpi$ on the generators
 $\{e_{i,0},f_{i,0},h_{i,\pm 1}\}_{i\in [n]}$ was computed in~\cite[Proposition 1.4]{T2}.
 Taking the $q\to 1$ limit in these formulas, we get
  $$\bar{\varpi}: E_{i,i+1}\otimes 1\mapsto E_{i,i+1}\otimes 1,\ E_{i+1,i}\otimes 1\mapsto E_{i+1,i}\otimes 1,$$
  $$\bar{\varpi}: E_{n,1}\otimes Z\mapsto (-d)^nE_{n,1}\otimes D,\ E_{1,n}\otimes Z^{-1}\mapsto (-d)^{-n} E_{1,n}\otimes D^{-1},$$
  $$\bar{\varpi}: (E_{i,i}-E_{i+1,i+1})\otimes D^{\pm 1}\mapsto d^{\mp n} (E_{i,i}-E_{i+1,i+1})\otimes Z^{\mp 1}$$
 for all $1\leq i\leq n-1$. Therefore, images of the elements
   $$E_{i,i+1}\otimes 1,\ E_{i+1,i}\otimes 1,\ E_{n,1}\otimes Z,\ E_{1,n}\otimes Z^{-1},\
     (E_{i,i}-E_{i+1,i+1})\otimes D^{\pm 1},\ c^{(1)}_\dd,\ c^{(2)}_\dd$$
 under $\bar{\varpi}$ and $\wt{\varpi}$ coincide.
  This completes our proof, since these elements generate $\bar{\dd}^{(n),0}_{d^n}$.
\end{proof}


\subsection{Classical limit of the commutative subalgebras $\A(\bar{s})$}
$\ $

  Let $\U^{(n),+}_{q,d}$ be the subalgebra of $\U^{(n)}_{q,d}$ generated by $\{e_{i,k}\}_{i\in [n]}^{k\in \ZZ}$.
 In~\cite{FT}, we introduced certain `large' commutative subalgebras
 $\A(\bar{s})$ of $\U^{(n),+}_{q,d}$ via the shuffle realization $\Psi:\U^{(n),+}_{q,d}\iso S$.
  We refer the interested reader to~\cite{FT} for a definition of
 the shuffle algebra $S$ and its subalgebras $\A(\bar{s})$, where
 $\bar{s}=(s_0,s_1,\ldots,s_{n-1})\in (\CC^\times)^{[n]}$
 satisfy $s_0s_1\cdots s_{n-1}=1$ and are \emph{generic}.
 Let $\diag_n\subset \MM_n$ be the subspace of diagonal matrices.

\begin{prop}\label{consequence 4}
  For $d\ne \sqrt{1}$ and a \emph{generic} $\bar{s}=(s_0,\ldots,s_{n-1})$
 satisfying $s_0\cdots s_{n-1}=1$, the isomorphism $\theta^{(n)}_d$ identifies
 the $q\to 1$ limit of $\A(\bar{s})$ with the universal enveloping algebra
 of the commutative Lie subalgebra $\underset{k>0}\bigoplus \diag_n\otimes Z^k$
 of $\bar{\dd}^{(n),0}_{d^n}$.
\end{prop}

\begin{proof}[Proof of Proposition~\ref{consequence 4}]
$\ $

  According to the main result~\cite[Theorem 3.3]{FT}, the algebra $\A(\bar{s})$ is a
 polynomial algebra in the generators $\{F'_{i,k}\}_{0\leq i\leq n-1}^{k\in \NN}$,
 where $F'_{i,k}$ is the coefficient of $(-\mu)^{n-i}$ in $F_k^\mu(\bar{s})$ defined via
  $$F_k^\mu(\bar{s}):=\frac
    {\prod_{i\in [n]} \prod_{1\leq j\ne j'\leq k}
     (x_{i,j}-q^{-2}x_{i,j'})\cdot \prod_{i\in [n]}(s_0\cdots s_i\prod_{j=1}^k x_{i,j}-\mu\prod_{j=1}^k x_{i+1,j})}
    {\prod_{i\in [n]} \prod_{1\leq j,j'\leq k} (x_{i,j}-x_{i+1,j'})}\in S_{k\delta}.$$

  First, we compute the $q\to 1$ limit of $\A(\bar{s})_{\delta}$.
 Choose $\beta_1\in \ZZ$ such that the $q\to 1$ limit of
 $(q-1)^{\beta_1}\cdot F'_{0,1}$ is well-defined and is
 non-zero\footnote{\ According to~\cite[Lemma 3.4]{T2}, we have $\beta_1=n-1$.}.
 Define $F_{i,1}:=(q-1)^{\beta_1}F'_{i,1}$ and let
 $\bar{F}_{i,1}$ denote the $q\to 1$ limit of $F_{i,1}$ (if it exists).
  According to~\cite[Corollary 3.12]{FT}, the element $F_{0,1}$ is a non-zero multiple
 of the first generator $h^{h}_1$ of the Heisenberg subalgebra $\h^{h}$.
 Combining this with Lemmas~\ref{consequence 2} and~\ref{consequence 3}, we see that
 $\theta^{(n)}_d(\bar{F}_{0,1})=\mu_1\cdot I_n\otimes Z$ for some $\mu_1\in \CC^\times$.

  For $1\leq i\leq n$, define
 $a_i:=s_0\cdots s_{i-1}\in \CC^\times,\
  A_i(d):=\sum_{j=1}^n d^{1-n\delta_{j,i}} E_{j,j}\in \MM_n$,
 and let $e_i(y_1,\ldots, y_n)$ be the $i$th elementary
 symmetric function in the variables $\{y_j\}_{j=1}^n$.

\begin{lem}\label{consequence 5}
 (a) The limit $\bar{F}_{i,1}$ is well-defined and
 $\theta^{(n)}_d(\bar{F}_{i,1})=\mu_1 e_i(a_1A_1(d),\ldots,a_nA_n(d))\otimes Z$.

\noindent
 (b) The limits $\{\bar{F}_{i,1}\}_{i=0}^{n-1}$ are linearly independent and
 $\{\theta^{(n)}_d(\bar{F}_{i,1})\}_{i=0}^{n-1}$ span $\diag_n\otimes Z$.
\end{lem}

\begin{proof}[Proof of Lemma~\ref{consequence 5}]
$\ $

 (a) It suffices to show that the image of the $q\to 1$ limit of
 $\frac{x_{i-1,1}}{x_{i,1}}F_{0,1}$ under $\theta^{(n)}_d$ equals
 $\mu_1 A_i(d)\otimes Z$. Recall the elements
 $\bar{h}'_{i,\pm 1}\in \mathrm{span}_\CC\langle \bar{h}_{0,\pm 1},\ldots, \bar{h}_{n-1,\pm 1}\rangle$
 from Lemma~\ref{shift operators} such that
 $[\bar{h}'_{i,1}, \bar{e}_{j,l}]=\delta_{i,j} \bar{e}_{j,l\pm 1}$
 for any $j\in [n], l\in \ZZ$.
 Since $\Psi(e_{j,l})=x_{j,1}^l$, we see that
 the $q\to 1$ limit of $\frac{x_{i-1,1}}{x_{i,1}}F_{0,1}$ equals
 $\ad(\bar{h}'_{i-1,1})\ad(\bar{h}'_{i,-1})(\bar{F}_{0,1})$.
  Combining the equality
 $$\theta^{(n)}_d(\bar{h}'_{i,\pm 1})=
   \left(\frac{d^{\pm (2n-i)}}{d^{\pm n}-1}(E_{1,1}+\cdots+E_{i,i})+
         \frac{d^{\pm (n-i)}}{d^{\pm n}-1}(E_{i+1,i+1}+\cdots+E_{n,n})\right)\otimes D^{\pm 1}$$
 with $\theta^{(n)}_d(\bar{F}_{0,1})=\mu_1 I_n\otimes Z$, we find
 $\theta^{(n)}_d(\ad(\bar{h}'_{i-1,1})\ad(\bar{h}'_{i,-1})(\bar{F}_{0,1}))=\mu_1 A_i(d)\otimes Z$
 as claimed.

 (b) Let $C(d)$ be an $n\times n$ matrix whose rows are the
 diagonals of $\{e_i(a_1A_1(d),\ldots,a_nA_n(d))\}_{i=0}^{n-1}$.
 If $d\ne \sqrt[n]{1}$ and $a_i\ne a_j\ \mathrm{for}\ i\ne j$
 (which is the case for generic $\bar{s}$),
 then $\det(C(d))\ne 0$ due to the Vandermonde determinant.
 The result follows.
\end{proof}

  Let us generalize the above result to $k>1$.
 According to~\cite[Theorems 3.2, 3.5]{T2}, we have
  $$\Psi\left(\exp\left(\sum_{r=1}^\infty a_r(d,q)\varpi(h^\perp_{0,r})c^{-r}\right)\right)=
    \sum_{k=0}^\infty (q-1)^{kn}b_k(d,q)F'_{0,k}c^{-k},$$
 where $c$ is a formal variable,
 the $q\to 1$ limits $\bar{a}_r(d)$ and $\bar{b}_k(d)$ of the constants
 $a_r(d,q)$ and $b_k(d,q)$ are nonzero for $d\ne 0$,
 $h^\perp_{0,r}\in \spa_\CC\langle h_{0,-r},\ldots,h_{n-1,-r}\rangle$ are defined via
 $\varphi(h^\perp_{0,r},h_{i,r})=\delta_{i,0}$ with the bilinear
 form $\varphi$ given by
 $\varphi(h_{i,-r},h_{j,s})=\delta_{r,s}\cdot \frac{b_n(i,j;-r)}{q-q^{-1}}$.
 Following our proof of Lemma~\ref{consequence 2}, we see that
 $h^\perp_{0,r}=(q-1)\lambda_r(d,q)h^v_{-r}$ and the $q\to 1$ limit of $\lambda_r(d,q)$ is nonzero.
 Combining this with Lemmas~\ref{consequence 2} and~\ref{consequence 3}, we find
  $\theta^{(n)}_d\left(q\to 1\ \mathrm{limit\ of}\ (q-1)^{-1}\varpi(h^\perp_{0,r})\right)=\bar{c}_r(d)\cdot I_n\otimes Z^r$,
 where $\bar{c}_r(d)\ne 0$ for $d\ne 0,\sqrt{1}$.
 Define $F_{i,k}:=(q-1)^{kn-1}F'_{i,k}$ and let $\bar{F}_{i,k}$
 denote the $q\to 1$ limit of $F_{i,k}$ (if it exists). We also set
 $\mu_r:=\bar{a}_r(d)\bar{c}_r(d)/\bar{b}_r(d)\in \CC^\times$.

 The above discussion implies that
 $\theta^{(n)}_d(\bar{F}_{0,k})=\mu_k\cdot I_n\otimes Z^k$ for any $k\in \NN$.


\begin{lem}\label{consequence 6}
  The limit $\bar{F}_{i,k}$ is well-defined and
 $\theta^{(n)}_d(\bar{F}_{i,k})=\mu_k e_i(a_1A_1(d^k),\ldots,a_nA_n(d^k))\otimes Z^k$.
 Moreover, the elements $\{\theta^{(n)}_d(\bar{F}_{i,k})\}_{i=0}^{n-1}$
 are linearly independent and span $\diag_n\otimes Z^k$.
\end{lem}

\begin{proof}[Proof of Lemma~\ref{consequence 6}]
$\ $

  To prove the first statement, it suffices to show
\begin{equation}\label{one}
 \theta^{(n)}_d\left(q\to 1\ \mathrm{limit\ of}\
 \frac{\prod_{j=1}^k x_{i-1,j}}{\prod_{j=1}^k x_{i,j}}F_{0,k}\right)=
 \mu_k A_i(d^k)\otimes Z^k\
 \mathrm{for\ any}\ 1\leq i \leq n.
\end{equation}
  Recall the elements
 $\bar{h}'_{i,\pm k}\in \mathrm{span}_\CC\langle \bar{h}_{0,\pm k},\ldots,\bar{h}_{n-1,\pm k}\rangle$
 from Lemma~\ref{shift operators} such that
 $[\bar{h}'_{i,\pm k}, \bar{e}_{j,l}]=\delta_{i,j} \bar{e}_{j,l\pm k}$ for any $j\in [n], l\in \ZZ$
 and the polynomials $P_{k,k}$ introduced in our proof of Theorem~\ref{main 1}.
  Define $L_{i;\pm k}\in \mathrm{End}(\ddot{u}^{(n),\geq}_{d})$ via
 $L_{i;\pm k}=P_{k,k}(\ad(\bar{h}'_{i,\pm 1}),\ldots, \ad(\bar{h}'_{i,\pm k}))$.
 Then, the $q\to 1$ limit of $\prod_{j=1}^k \frac{x_{i-1,j}}{x_{i,j}}\cdot F_{0,k}$
 equals $L_{i-1;k}L_{i;-k}(\bar{F}_{0,k})$.
  To derive~(\ref{one}), one needs to apply the formula
 $$\theta^{(n)}_d(\bar{h}'_{i,\pm k})=
   \left(\frac{d^{\pm (2n-i)k}}{d^{\pm nk}-1}(E_{1,1}+\cdots+E_{i,i})+
   \frac{d^{\pm (n-i)k}}{d^{\pm nk}-1}(E_{i+1,i+1}+\cdots+E_{n,n})\right)\otimes D^{\pm k}$$
 together with the identity
 $P_{k,k}(\frac{d^{kn}-1}{d^n-1}, \frac{d^{2kn}-1}{d^{2n}-1},\cdots, \frac{d^{k^2n}-1}{d^{kn}-1})
  =e_k(1,d^n,\ldots,d^{(k-1)n})=d^{\frac{k(k-1)n}{2}}$.

  The linear independence of $\theta^{(n)}_d(\bar{F}_{i,k})$ is proved completely
 analogously to Lemma~\ref{consequence 5}(b).
\end{proof}

  It remains to note that Proposition~\ref{consequence 4} follows
 from Lemma~\ref{consequence 6} by induction on $k$.
\end{proof}



\medskip

\begin{thebibliography} {XXX}

\bibitem[BGK]{BGK}
  S.~Berman, Y.~Gao, Y.~Krylyuk,
    {\em Quantum tori and the structure of elliptic quasi-simple Lie algebras},
  J. Funct. Anal. {\bf 135} (1996), no.~2, 339--389.

\bibitem[BKLY]{BKLY}
 C.~Boyallian, V.~Kac, J.~Liberati, C.~Yan,
   {\em Quasifinite highest weight modules over the Lie algebra of matrix differential operators on the circle},
 J. Math. Phys. {\bf 39} (1998), no. 5, 2910--2928.

\bibitem[FJMM]{FJMM}
  B.~Feigin, M.~Jimbo, T.~Miwa, E.~Mukhin,
    {\em Representations of quantum toroidal $\gl_n$},
  J. Algebra {\bf 380} (2013), 78--108; arXiv:1204.5378.

\bibitem[FT]{FT}
  B.~Feigin, A.~Tsymbaliuk,
    {\em Bethe subalgebras of $U_q(\widehat{\gl}_n)$ via shuffle algebras},
  Sel. Math. {\bf 22} (2016), no.~2, 979--1011; arXiv:1504.01696.

\bibitem[M]{M}
  K.~Miki,
    {\em Toroidal braid group action and an automorphism of toroidal algebra $U_q(\ssl_{n+1,tor})\ (n\geq 2)$},
  Lett. Math. Phys. {\bf 47} (1999), no.~4, 365--378.

\bibitem[MRY]{MRY}
  R.~V.~Moody, S.~E.~Rao, T.~Yokonuma,
    {\em Toroidal Lie algebras and vertex representations},
  Geom. Dedic. {\bf 35} (1990), no. 1--3, 283--307.

\bibitem[T1]{T1}
  A.~Tsymbaliuk,
    {\em The affine Yangian of $\gl_1$ revisited},
  Adv. Math. {\bf 304} (2017), 583--645; arXiv:1404.5240.

\bibitem[T2]{T2}
  A.~Tsymbaliuk,
    {\em Several realizations of Fock modules for quantum toroidal algebras of $\ssl_n$},
  preprint, arXiv:1603.08915.

\bibitem[TB]{TB}
  A.~Tsymbaliuk, M.~Bershtein,
    {\em Homomorphisms between different quantum toroidal and affine Yangian algebras},
  preprint, arXiv:1512.09109.

\bibitem[VV]{VV}
  M.~Varagnolo, E.~Vasserot,
    {\em Double-loop algebras and the Fock space},
  Invent. Math. {\bf 133} (1998), no.~1, 133--159; arXiv:q-alg/9612035.
\end{thebibliography}
\end{document}